\documentclass[a4,pdflatex]{amsart}

\RequirePackage[T2A,T1]{fontenc}
\RequirePackage{amsmath}
\RequirePackage{amssymb}%
\RequirePackage{amsthm}%
\RequirePackage{amsfonts}
\RequirePackage{txfonts}
\RequirePackage{calc}
\RequirePackage{units}
\RequirePackage{subfigure}
\RequirePackage{color}
\RequirePackage[font={footnotesize,sf,sl}]{caption}
\RequirePackage{bbm}
\RequirePackage{tikz}
\usetikzlibrary{calc}
\usetikzlibrary{decorations.pathreplacing}
\usetikzlibrary{arrows,shapes}

\tikzstyle{perspective nqgs}=[%
x={(.8,.2)},%
y={(.3,1.7)},%
z={(2.3,-.3)}]

\definecolor{col1}{RGB}{233,80,62}
\definecolor{col2}{RGB}{0,136,119}
\definecolor{col3}{RGB}{185,15,34}
\definecolor{col4}{RGB}{36,53,114}

\RequirePackage[colorlinks,linkcolor=col4, 
                citecolor=col4,
                filecolor=col4,urlcolor=col4,pdfdisplaydoctitle,
                anchorcolor=col4,menucolor=col4,
                bookmarksopen=false,bookmarks=false]{hyperref}

\hypersetup{
pdfauthor={Andreas Paffenholz, TU Darmstadt, Fachbereich Mathematik, Dolivostr. 15, 64293 Darmstadt, Germany},
pdftitle={Finiteness of the Polyhedral Q-Codegree Spectrum},
pdfsubject={Finiteness of the Polyhedral Q-Codegree Spectrum},
pdfkeywords={lattice polytopes, toric varieties, Q-codegree, spectrum conjecture},
pdfcreator={Andreas Paffenholz}
}

\newcommand{\Q}{\mathbb Q}
\newcommand{\Z}{\mathbb Z}
\newcommand{\R}{\mathbb R}
\newcommand{\N}{\mathbb N}
\newcommand{\va}{\mathbf a}
\newcommand{\vx}{\mathbf x}
\newcommand{\vy}{\mathbf y}

\newcommand{\vu}{\mathbf u}
\newcommand{\vv}{\mathbf v}
\newcommand{\vw}{\mathbf w}

\newcommand{\ve}{\mathbf e}
\newcommand{\0}{\mathbf 0}
\newcommand{\1}{\mathbf 1}
\newcommand{\ov}[1]{\overline{#1}}
\newcommand{\eps}{\varepsilon}
\newcommand{\vxcore}{\vx_{\core}}
\newcommand{\spectrum}{\mathcal S}
\newcommand{\speccan}[1][\alpha]{\mathcal S^{\mathsf{can}}_{#1}}
\newcommand{\pclass}{\mathcal P}
\newcommand{\pccan}[1][\alpha]{\mathcal P^{\mathsf{can}}_{#1}}
\newcommand{\nclass}{\mathcal N}
\newcommand{\cf}{\mathcal{N}_{\core}}
\newcommand{\setA}{{\mathcal A}}
\newcommand{\adjP}[2][c]{#2^{(#1)}}
\newcommand{\sprod}[1]{\left\langle\,#1\,\right\rangle}

\DeclareMathOperator{\relint}{relint}
\DeclareMathOperator{\cone}{cone}
\DeclareMathOperator{\aff}{aff}
\DeclareMathOperator{\conv}{conv}
\DeclareMathOperator{\qcd}{codeg_\Q}
\DeclareMathOperator{\core}{core}
\DeclareMathOperator{\Caff}{C_{\mathsf{aff}}}
\DeclareMathOperator{\Acore}{A_{\core}}
\DeclareMathOperator{\mountain}{mountain}
\DeclareMathOperator{\height}{height}
\DeclareMathOperator{\rank}{rank}

\theoremstyle{plain}
\newtheorem{theorem}{Theorem}[section]

\newtheorem{corollary}[theorem]{Corollary}
\newtheorem{lemma}[theorem]{Lemma}

\theoremstyle{definition}
\newtheorem{remark}[theorem]{Remark}
\newtheorem{example}[theorem]{Example}
\newtheorem{conjecture}[theorem]{Conjecture}

\date{\today}

\title{Finiteness of the Polyhedral $\Q$-Codegree Spectrum}
\author{Andreas Paffenholz}
\thanks{The author is supported by the Priority Program 1489 of the German Research Council (DFG)}
\address{TU Darmstadt, Fachbereich Mathematik, Dolivostr. 15, 64293 Darmstadt, Germany}
\email{paffenholz@mathematik.tu-darmstadt.de}

\begin{document}

\begin{abstract}
  In this paper we show that the spectrum of the $\Q$-codegree of a
  lattice polytope is finite above any positive threshold in the class
  of lattice polytopes with $\alpha$-canonical normal fan for any
  fixed $\alpha>0$. For $\alpha=1/r$ this includes lattice polytopes
  with $\Q$-Gorenstein normal fan of index $r$.  In particular, this
  proves Fujita's Spectrum Conjecture for polarized varieties in the
  case of $\Q$-Gorenstein toric varieties of index $r$.
\end{abstract}

\maketitle

\numberwithin{equation}{section}
\numberwithin{figure}{section}

\section{Introduction}

Let $P \subseteq \R^d$ be a $d$-dimensional rational polytope given by
\begin{align}
  P\ =\ \left\{\; \vx\,\mid\, \sprod{\va_i,\vx}\ \le\ b_i\,,\quad 1\le
    i\le n\; \right\}\,\label{eq:ineq}
\end{align}
with \emph{primitive} linear functionals $\va_i\in(\Z^d)^*$ for $1\le
i\le n$, where a functional is primitive if the greatest common
divisor of its entries is $1$. We may assume that this representation
is \emph{irredundant}, \textit{i.e.}\ no inequality $\sprod{\va_i,
  \vx}\le b_i$ can be omitted without changing $P$. The polytope $P$
is a \emph{lattice polytope} if all its vertices are integral.

Let $X$ be a normal projective algebraic variety. A line bundle $L$ on
$X$ is \emph{ample} if it has positive intersection with all
irreducible curves on $X$, and $L$ is \emph{big} if the global section
of some multiple define a birational map to some projective space. Any
ample line bundle is big.  A \emph{polarized variety} is a pair
$(X,L)$ of a normal projective algebraic variety $X$ and an ample line
bundle $L$ on $X$. Let $K_X$ be the canonical class on $X$. For a
rational parameter $c>0$ the \emph{adjoint line bundles}
on $X$ are the line bundles $L + c\cdot K_X$. They define
invariants that have been used for classifications of projective toric
varieties. In particular, the \emph{unnormalized spectral value}
$\mu(L)$ of the polarized variety is given by
\begin{align*}
  \mu(L)^{-1}\ &:= \ \sup \left(c \in \Q \;\middle|\; 
    L\, +\,c\cdot K_X \text{ is big} \right)\,,
\end{align*}
see~\cite[Ch.~7.1.1]{BS95}. Note that $\mu(L)<\infty$ as $L$ is
big. The \emph{spectral value} $\sigma(L):=d+1-\mu(L)$ has originally
been considered by Sommese~\cite[\S 1]{MR855885}. Similar notions have
been defined several times, \emph{e.g,} $\kappa\eps(X,L):=-\mu(L)$ is
the \emph{Kodeira energy} of $(X,L)$.  Fujita has stated the following
conjecture on the values of $\mu(L)$.
\begin{conjecture}[{Spectrum Conjecture, Fujita~\cite{Fuj92} and
    \cite[(3.2)]{Fuj96}}]\label{conj:fujita} For any $d\in \N$ let
    $S_d$ be the set of unnormalized spectral values of a smooth
    polarized $d$-fold. Then, for any $\eps>0$, the set $\{s\in
    S_d\mid s>\eps\}$ is a finite set of rational numbers.
\end{conjecture}
This has been proved by Fujita for dimensions $d=2,3$ in
1996~\cite{Fuj96}. Recently, Cerbo~\cite{1210.5324} proved the related
log spectrum conjecture~\cite[(3.7)]{Fuj96}.

There is a fundamental connection between combinatorial and algebraic
geometry.  Any lattice polytope $P$ uniquely defines a pair
$(X_P,L_P)$, where $X_P$ is a projective toric variety polarized by an
ample line bundle $L_P$, and vice versa (see,
\textit{e.g.},~\cite{Ful93}). There is a close connection between
combinatorial and algebraic notions.  Using this correspondence we
give a polyhedral interpretation of the spectral value in the case of
projective toric varieties.

Let $P$ be a lattice polytope given as in \eqref{eq:ineq} by an
irredundant set of inequalities with primitive normals. The family of
\emph{adjoint polytopes} associated to $P$ is given by
\begin{align}
  \adjP[c] P\ =\ \left\{\; \vx\,\mid\, \sprod{\va_i,\vx}\ \le\
    b_i\,-\,c\,,\quad 1\le i\le n\; \right\}\label{eq:adjoint}
\end{align}
for a rational parameter $c>0$.  This notion has been introduced by
Dickenstein, Di Rocco, and Piene~\cite{DDP09}; see
also~\cite{1105.2415}.  If $P$ is the polytope associated to a
polarized projective toric variety $(X_P,L_P)$, then $\adjP[q]{(pP)}$
for some integral $p,q>0$ is the polytope that corresponds to the the
line bundle $p\cdot L+q\cdot K_X$. Clearly $\adjP P$ is empty for
large $c$. Let $c=\nicefrac qp>0$ for $p,q\in\Z_{>0}$ be the maximal
rational number such that $\adjP P$ (or, equivalently,
$\adjP[q]{(pP)}$) is non-empty.  Its reciprocal, the $\Q$-codegree
$\qcd\, P \, :=\, c^{-1}\, =\, \nicefrac pq$ of $P$ equals precisely
$\mu(L)$, see ~\cite{DDP09,1105.2415}.  We can now reformulate
Conjecture~\ref{conj:fujita} in the case of projective polarized toric
varieties.
\begin{conjecture}[Polyhedral Spectrum
  Conjecture]\label{conj:fujitapolyhedral}
  For any $d\in \N$ let $S_d$ be the set of $\Q$-codegrees of a
  lattice polytope corresponding to a polarized smooth toric
  variety. Then, for any $\eps>0$, the set $\{s\in S_d\mid s>\eps\}$
  is a finite set of rational numbers.
\end{conjecture}
The purpose of this note is to show that this conjecture is even true
on the much larger class of lattice polytopes with $\alpha$-canonical
normal fan (Theorem~\ref{thm:spectrumcanonical}). This class includes
the set of lattice polytopes with $\Q$-Gorenstein normal fan of index
$r$ (for $\alpha=\nicefrac1r$), and in particular, for $\alpha=1$,
contains all smooth polytopes (see below for definitions).

A polytope $P$ with $\Q$-Gorenstein normal fan corresponds to a
polarized $\Q$-Gorenstein toric variety $(X,L)$ of index $r$,
\textit{i.e.}\ the integer $r\in \N$ is the minimal $r$ such that
$rK_X$ is a Cartier divisor. Thus, we prove
Conjecture~\ref{conj:fujita} for the class of $\Q$-Gorenstein
varieties of index $r$ (Corollary~\ref{cor:algebraic}), which contains
all smooth polytopes (for $r=1$).

\subsubsection*{Acknowledgments}
I learned about the classical spectrum conjecture of Fujita through
joint work with Sandra Di Rocco, Christian Haase, and Benjamin Nill
on~\cite{1105.2415}, and I would like to thank them for several
discussions on polyhedral adjunction theory. Christian Haase gave
valuable comments at an early stage of this work. Silke Horn and
Michael Joswig suggested several improvements for the exposition of
this paper.

\section{Basic Definitions}

Let $P$ be a $d$-dimensional lattice polytope given by an irredundant
set of primitive normal vectors $\va_i$, $1\le i\le n$ and
corresponding right hand sides $b_i$, $1\le i\le n$ as in
\eqref{eq:ineq}. Here, a set of normal vectors is \emph{irredundant}
if we cannot omit one of the inequalities $\sprod{\va_i,\vx}\le b_i$
without changing $P$, and the integral vector $\va_i$ is
\emph{primitive} if there is no other integral point strictly between
$\va_i$ and the origin, for any $1\le i\le n$. Up to a lattice
preserving transformation we can always assume that $P$ is
full-dimensional, so $P\subseteq\R^d$. Let $\Sigma_P\subseteq(\R^d)^*$
be the \emph{normal fan} of $P$. This is a complete rational
polyhedral fan. With $\Sigma_P(k)$ we denote the subset of $\Sigma_P$
of cones of dimension at most $k$ (the \emph{$k$-skeleton}). The set
of normals $\va_1, \ldots, \va_n$ is irredundant, so they generate
the rays in $\Sigma_P(1)$.  For a rational parameter $c\ge 0$ the
family of \emph{adjoint polytopes} $\adjP[c] P$ of $P$ is given by
\eqref{eq:adjoint}.
\begin{remark}
  Note that it is essential for the definition of the adjoint polytope
  that all inequalities $\sprod{\va_i,\vx}\le b_i$ define a facet of
  $P$ (\textit{i.e.}, the system is irredundant), and that all $\va_i$
  are primitive vectors, for $1\le i\le n$. For example, consider the
  triangle defined by
  \begin{align*}
    x\ &\ge\ 0 & y\ &\ge \ 0& 3x+y\ &\le \ 3\,.
  \end{align*}
  Then $\adjP P\ne \varnothing$ for $c\le \frac 35$ and empty
  otherwise. However, if we add the redundant inequality $x\le 1$,
  then $\adjP P=\varnothing$ for $c>\frac 12$. If we replace the last
  inequality by $6x+2y\le 6$, then $\adjP P$ is non-empty for any
  $c\le \frac 23$.
\end{remark}
For sufficiently large $c$ the adjoint polytopes $\adjP P$ are
empty. The \emph{$\Q$-codegree} $\qcd(P)$ of $P$ is the inverse of the
largest $c>0$ such that $\adjP P$ is non-empty, \textit{i.e.}
\begin{align*}
  \qcd(P)^{-1}\ :=\ \max\left(c\mid \adjP P\ne
    \varnothing\right)\ =\ \sup\left( c\mid \dim\, \adjP P = d
  \right)\,.
\end{align*}
See \cite[Def.~1.5 and Prop.~1.6]{1105.2415} for a proof that these
two definitions of the $\Q$-codegree coincide, and for more
background.  The \emph{core} of $P$ is
\begin{align*}
  \core\,P\ :=\ \adjP[\qcd(P)] P\,.
\end{align*}
The core of $P$ is a (rational, not full dimensional) polytope defined
by the (usually redundant set of) inequalities
$\sprod{\va_i,\vx}\,\le\, b_i-\qcd(P)$. A vector $\va_i$ is a
\emph{core normal} if
\begin{align*}
  \sprod{\va_i,\vy}\,&=\,b_i-\qcd(P)\ \text{ for all }\  \vy\in \core\,P\,.
\end{align*}
The set of all \emph{core normals} of $P$ will be denoted by
$\cf(P)$. It is subset of the primitive generators of the rays in
$\Sigma_P(1)$ (but they do not necessarily span a subfan of
$\Sigma_P$).  Up to relabeling we can assume that
$\cf(P)\,=\,\{\,\va_1, \ldots, \va_m\,\}$ for some $m\le n$.  The
\emph{affine hull} $\Caff(P)$ of $\core\, P$ is the smallest affine
space that contains $\core\, P$. This is the set of solutions of all
equations given by the core normals, \emph{i.e.}
\begin{align*}
  \Caff(P)\ &:=\ \left\{\vx\mid \sprod{\va_i, \vx}=b_i-\qcd(P)\,,\quad 1\le
  i\le m\right\}\,.
\end{align*}
\begin{example}\label{ex:twodim}
  Consider the polytope shown in Figure~\ref{fig:core} given by the
  five inequalities
  \begin{align*}
    -y\ &\le\ 0& -x\ &\le\ 0& x-y\ &\le\ 4\\
    y\ &\le\ 3& x\ &\le\ 5\,.\\
  \intertext{The core of $P$ is the segment  with
  vertices $(\nicefrac32,\nicefrac32)$ and $(\nicefrac72,\nicefrac32)$
  given by the inequalities}
    -y\ &\le\ -\nicefrac32\ = \ 0\,-\, \nicefrac32& -x\ &\le\ -\nicefrac32\ = \ 0\,-\, \nicefrac32
    &x-y\ &\le\ \phantom{-}\nicefrac52\ = \ 4\,-\, \nicefrac32\\
    y\ &\le\ \phantom{-}\nicefrac32\ = \ 3\,-\, \nicefrac32& x\ &\le\
    \phantom{-}\nicefrac72\ = \ 5\,-\, \nicefrac32\,.
  \end{align*}
  It is drawn with a thick line in the figure.  The inequalities in
  the left column actually span the affine hull of $\core\, P$, so
  $\cf(P)=\{(0,\pm1)\}$. The inequality in the right column is
  redundant. We have $\Caff=\{(x,y)\mid y=\nicefrac32\}$.
\end{example}
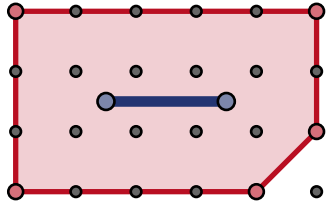
\begin{figure}[t]
  \centering
  \begin{tikzpicture}[scale=.4]
    \foreach \x in {0,...,10} { 
      \foreach \y in {0,...,6} { 
        \coordinate (a\x\y) at (\x,\y); 
      } 
    } 

    \tikzstyle{dot} = [draw,line width=1pt,black,fill=black!60] 
    \tikzstyle{vert} = [draw,line width=1pt,black,fill=col3!60] 
    \tikzstyle{cvert} = [draw,line width=1pt,black,fill=col4!60] 
    \tikzstyle{bdedge} = [draw,line width=2pt,color=col3,fill=col3!20] 
    \tikzstyle{inneredge} = [draw,line width=4pt,col4] 
    \draw[bdedge] (a00) -- (a80) -- (a102) -- (a106) --  (a06) -- cycle;
    \foreach \x in {0,2,4,6,8,10} { 
      \foreach \y in {0,2,4,6} { 
        \draw[dot] (a\x\y) circle (5pt);
      } 
    } 
    \foreach \x/\y in {0/0,0/6,8/0,10/2,10/6} {
      \draw[vert] (a\x\y) circle (7pt);
    }
    \draw[inneredge] (a33) -- (a73);
    \draw[cvert] (a33) circle (8pt);
    \draw[cvert] (a73) circle (8pt);
  \end{tikzpicture}
  \caption{The lattice polytope of example~\ref{ex:twodim}. Its core
    is the segment drawn with a thick line.}
\label{fig:core}
\end{figure}

Let $\sigma\subset (\R^d)^*$ be a a rational polyhedral cone with
primitive generators $\va_1, \ldots, \va_k$. The cone $\sigma$ is
called \emph{$\Q$-Gorenstein} of index $r$ if there is a primitive
vector $\vu$ such that $\sprod{\va_i, \vu}=r$ for $1\le i\le
k$. The cone is \emph{Gorenstein} if $r=1$. A complete rational
polyhedral fan $\Sigma$ is \emph{$\Q$-Gorenstein} of index $r$ if all
cones are $\Q$-Gorenstein and $r$ is the least common multiple of the
indices of all cones. The fan $\Sigma$ is \emph{Gorenstein} if $r=1$.

We define the height of a point $\vy\in \sigma$ as
\begin{align*}
  \height(\vy)\ :=\ \max\left(\;\sum_{i=1}^k\lambda_i\;\middle|\;
    \vy=\sum_{i=1}^k\lambda_i\va_i,\ \text{ and } \lambda_i\ge 0\,\
    \text{ for } 1\le i\le
    k\;\right)\,.
\end{align*}
The cone $\sigma$ is \emph{$\alpha$-canonical} for some $\alpha>0$ if
$\height(\vy)\ge \alpha$ for any integral point
$\vy\in\sigma\cap(\Z^d)^*$.  A complete rational polyhedral fan
$\Sigma$ is \emph{$\alpha$-canonical} if every cone in $\Sigma$ is. A
cone or fan is \emph{canonical} if it is $\alpha$-canonical for
$\alpha=1$. Note that $1$-dimensional cones are always canonical, so
$\alpha\le 1$ for all fans. Any $\Q$-Gorenstein cone or fan of index
$r$ is $\frac 1r$-canonical. We define
\begin{flalign*}
  &&\pclass(d)\ &:=\ \left\{\; P\; \mid\; P \text{ is a
      $d$-dimensional lattice polytope}\;\right\}\,\\
  \text{and}&&\hfill\pccan(d)\ &:=\ \left\{\; P\; \mid\; P \,\in\,\pclass(d) \text{ with
      $\alpha$-canonical normal fan}\;\right\}\,.&&\hfill
\end{flalign*}
A lattice polytope is \emph{smooth} if the primitive generators of any
cone in its normal fan are a subset of a lattice basis. So, in
particular, the normal fan of a smooth polytope is Gorenstein (with
index $1$) and canonical. 

By the toric dictionary, a polarized toric variety $(X,L)$ is
non-singular if and only if the associated lattice polytope is smooth
(see \emph{e.g.}~\cite{Ful93}). The polarized variety $(X,L)$ is
$\Q$-Gorenstein if $rK_X$ is a Cartier divisor for some integer
$r$. The minimal $r$ such that this holds is the \emph{index} of the
polarized toric variety. Again, $(X,L)$ is $\Q$-Gorenstein of index
$r$ if the normal fan of the associated polytope is $\Q$-Gorenstein of
index $r$.

\section{The $\Q$-Codegree Spectrum}
\label{sec:spectrum}

The purpose of this note is to study the set of lattice polytopes with
$\alpha$-canonical normal fan whose $\Q$-codegree is bounded from
below. For $\alpha, \eps>0$ and $d\in\Z_{>0}$ we define
\begin{flalign*}
  &&\spectrum(d,\eps)\ &:=\ \left\{\; P\; \mid \; P\in \pclass(d)\text{
      and } \qcd(P)\ge \eps\; \right\}\\
  \text{and}&\qquad&\speccan(d,\eps)\ &:=\ \left\{\; P\; \mid \; P\in
    \pccan(d) \text{ and } \qcd(P)\ge \eps\;
  \right\}\ =\
    \spectrum(d,\eps)\;\cap\; \pccan(d)\;.
\end{flalign*}
For $\alpha=1$, this set contains all smooth lattice polytopes, and,
for $\alpha=\nicefrac1r$, all lattice polytopes with $\Q$-Gorenstein
normal fan of index $r$ and $\Q$-codegree bounded from below by
$\eps$.  Our main theorem is the following.
\begin{theorem}\label{thm:spectrumcanonical}
  Let $d\in \N$ and $\alpha, \eps\ge 0$ be given. Then
  \begin{align*}
    \left\{\,\qcd(P)\,\mid \, P\in\speccan(d,\eps)\,\right\}
  \end{align*}
  is finite.
\end{theorem}
In other words, in any fixed dimension and for any fixed $\alpha$, the
set of values the $\Q$-codegree assumes above any positive threshold
is finite in the class of polytopes with $\alpha$-canonical normal
fan. We obtain the following corollary.
\begin{corollary}
  Conjecture~\ref{conj:fujitapolyhedral} holds for $d$-dimensional
  lattice polytopes with $\alpha$-canonical normal fan, for any given
  $\alpha>0$.
\end{corollary} 
Using the correspondence between combinatorial and toric geometry we
can translate this into a statement about polarized toric varieties.
\begin{corollary}\label{cor:algebraic}
  Conjecture~\ref{conj:fujita} holds for $d$-dimensional polarized
  $\Q$-Gorenstein toric varieties of index $r$, for any integer $r>0$.
\end{corollary}
This proves a generalization of the Conjecture~\ref{conj:fujita} in
the toric case. For smooth two and three dimensional polarized
varieties this has previously been proved by Fujita~\cite{Fuj96}.  
\begin{remark}
  The following family of examples shows that we cannot expect the
  conjecture to be true without any further assumptions. For positive
  integers $a\in \Z_{>0}$ consider the polytopes
\begin{align*}
  \Delta_d(a)\ :=\ \conv(\0, a\ve_1, \ve_2, \ldots, \ve_d)\,.
\end{align*}
They have $\Q$-codegree $\qcd(\Delta_d(a))\,=\, d-1+\frac 2a$
arbitrarily close to $d-1$, and the normal fan of $\Delta(a)$ is
$\Q$-Gorenstein with index $a$. Hence, the theorem cannot hold without
restrictions to the normal fan unless $\eps>d-1$.
\end{remark}
We prove Theorem~\ref{thm:spectrumcanonical} by a series of Lemmas.
The proof has two main steps. We first show that there are, up to
lattice equivalence and for fixed $\alpha>0$ and $d\in\Z_{>0}$, only
finitely many sets of core normals for polytopes in the class
$\pccan(d)$ (Corollary~\ref{cor:spectrumcanonical}) by reducing this
to a finiteness result of Lagarias and Ziegler
(Theorem~\ref{thm:lagariasZiegler}).  In a second step we show that
each such configuration gives rise to a finite number of different
$\Q$-codegrees above any positive threshold
(Lemma~\ref{prop:spectrumcanonical}).

We start by studying the set of normal vectors that define the core of
a lattice polytope. Let $P$ be a $d$-dimensional lattice polytope with
$\alpha$-canonical normal fan given by
\begin{align*}
  P\ =\ \left\{\; \vx\,\mid\, \sprod{\va_i,\vx}\ \le\ b_i\,\quad 1\le
    i\le n\; \right\}
\end{align*}
with an irredundant set of primitive normals $\va_i$. Let $C=\core\,P$
be the core of $P$ with affine hull $\Caff=\aff\, C$ and $c^{-1}$ the
$\Q$-codegree of $P$. Up to relabeling we can assume that $\va_1,
\ldots, \va_m$ for some $m\le n$ is the (not irredundant) set of facet
normals defining $\Caff$, \textit{i.e.},
\begin{align*}
  \Caff\ =\ \{\,\vx\mid \sprod{\va_i,\vx}\;=\; b_i-c\,,\quad 1\le i\le m\,\}
\end{align*}
and no other $\va_i$ is constant on $\Caff$.  We define a new polytope
$\mountain\,P$ via
\begin{align*}
  \mountain(P)\ =\ \left\{\,(\vx,\lambda)\,\mid \,
    \sprod{\va_i,\vx}+\lambda\le b_i\quad \text{ for } 1\le i\le n\;
    \quad \text{ and }\quad \lambda\ge 0\,\right\}\,.
\end{align*}
See Figure~\ref{fig:mountain} for the mountain of
Example~\ref{ex:twodim}.
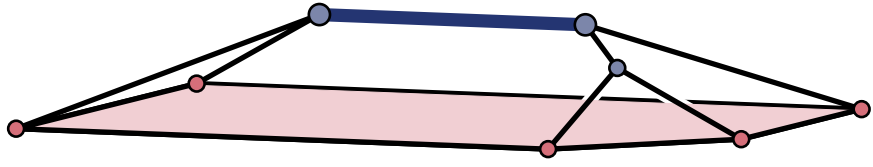
\begin{figure}[t]
  \centering
  \begin{tikzpicture}[perspective nqgs,scale=.5]
    \foreach \x in {0,...,10} { 
      \foreach \y in {0,...,6} { 
        \coordinate (a\x\y) at (\y,0,\x); 
      } 
    } 
    \coordinate (b33) at (3,1.5,3);
    \coordinate (b53) at (3,1.5,7);
    \coordinate (b64) at (2,1,8);

    \coordinate (a06a) at (5.8,0,0);
    \coordinate (a106a) at (5.8,0,10);

    \tikzstyle{dot} = [draw,line width=1pt,black,fill=col2!60] 
    \tikzstyle{vert} = [draw,line width=1pt,black,fill=col3!60] 
    \tikzstyle{cvert} = [draw,line width=1pt,black,fill=col4!60] 
    \tikzstyle{bdfilledge} = [draw,line width=0pt,color=col3!20,fill=col3!20,join=round] 
    \tikzstyle{bdbedge} = [draw,line width=2pt,color=black,join=round] 
    \tikzstyle{bdwfilledge} = [draw,line width=2pt,black,join=round] 
    \tikzstyle{bdedge} = [draw,line width=2pt,black,join=round] 
    \tikzstyle{bdwedge} = [draw,line width=5pt,white,join=round] 
    \tikzstyle{inneredge} = [draw,line width=5pt,col4] 
    \draw[bdwfilledge] (a00) -- (a80) -- (a102) -- (a106) --  (a06) -- cycle;
    \draw[bdwedge] (a80) -- (b64) -- (a102);

    \fill[bdfilledge] (a00) -- (a80) -- (a102) -- (a106a) --  (a06a) -- cycle;
    \draw[bdbedge] (a06) -- (a00) -- (a80) -- (a102) -- (a106);

    \draw[bdedge] (a80) -- (b64) -- (a102);

    \draw[bdedge] (b64) -- (b53) -- (a106);
    \draw[bdedge] (a00) -- (b33) -- (a06);
    \foreach \x/\y in {0/0,0/6,8/0,10/2,10/6} {
      \draw[vert] (a\x\y) circle (6pt);
    }

    \draw[inneredge] (b33) -- (b53);
    \draw[cvert] (b33) circle (8pt);
    \draw[cvert] (b53) circle (8pt);    
    \draw[cvert] (b64) circle (6pt);    
  \end{tikzpicture}
  \caption{The mountain of the polytope in Figure~\ref{fig:core}. The
    top (thick) face is the core of $P$, up to a projection.  \label{fig:mountain}}
\end{figure}
Up to a projection we can recover all adjoint polytopes by
intersecting the mountain with the hyperplane $\lambda=c$ for some
$c>0$, \textit{i.e.},
\begin{align*}
  \adjP P\;\times\;\{c\}\ =\ \mountain\, P\; \cap\;
  \{(\vx,\lambda)\mid \lambda=c\}\,.
\end{align*}
Let $F$ be the face of $\mountain\,P$ defined by the inequality
$\lambda\le c$ for some (appropriately chosen) $c\in \R$. Then
$\core\,P$ is the projection of $F$ onto the first $d$ coordinates,
and the $\Q$-codegree of $P$ is the inverse $\nicefrac 1c$ of the
height of this face over the base $P\times\{0\}$.

Let $\Acore:=\conv(\va_1, \ldots, \va_m)$ be the convex hull of the
core normals. The following two lemmas show that all $\va_i$ are
vertices and $\0$ is a relatively interior point of $\Acore$
(\textit{i.e.}, the vectors $\va_1, \ldots, \va_m$ positively span the
linear space they generate). See
Figure~\ref{fig:acoreandlinspan}\subref{fig:acore}. The first
lemma has been proved in~\cite{1105.2415}.
\begin{lemma}[{Di Rocco et al. 2011~\cite[Lemma 2.2]{1105.2415}}]\label{lem:zeroInInterior}
  The origin is in the relative interior of $\Acore$.
\end{lemma}
\begin{proof}
  We consider the mountain of $P$. The vector $(\0,1)$ defining the
  face $F$ projecting onto the core of $P$ as above is in the relative
  interior of the normal cone of $F$, and the generators of this
  normal cone are the normals defining $\Caff$.

  By construction, this normal cone is spanned by the vectors
  $(\va_i,1)$ for $1\le i\le m$, so there are non-zero, non-negative
  coefficients $\eta_1, \ldots, \eta_m$ such that
  \begin{align*}
    (\0,1)\ =\ \sum_{i=1}^m\eta_i(\va_i,1)\,.
  \end{align*}
  This reduces to a positive linear combination of $\0$ in the
  $\va_i$, $1\le i\le m$, which proves that $\0\in \relint\,\Acore$.
\end{proof}
\begin{lemma}\label{lem:vertices}
  The vertices of $\Acore$ are $\va_1, \ldots, \va_m$.
\end{lemma}
\begin{proof}
  Assume on the contrary that this is not the case. Then one of the
  $\va_j$, $1\le j\le m$ is a convex combination of the others. We can
  assume that this is $\va_m$. So there are $\eta_1, \ldots,
  \eta_{m-1}\ge 0$ such that $\va_m=\sum\eta_i\va_i$ and $\sum\eta_i=
  1$.  Let $\vxcore$ be a relative interior point of $\core(P)$. Then
  $\sprod{\va_j,\vxcore}\,=\, b_j-c$ for $1\le j\le m$, so we compute
  \begin{align*}
    -c+b_m\ &=\ \sprod{\va_m, \vxcore}\ =\ \textstyle\sum\eta_i\sprod{\va_i, \vxcore}\\
    &=\ \textstyle\sum\eta_ib_i-q\textstyle\sum\eta_i\ =\ -c+
    \textstyle\sum\eta_ib_i\,.
  \end{align*}
  Hence, $\sum\eta_ib_i=b_m$.  On the other hand, by irredundancy of
  the $\va_i$ as facet normals of $P$ there is $\vy$ (relative
  interior to the facet defined by $\va_m$) such that $\sprod{\va_m,
    \vy}\,=\, b_m$, but $\sprod{\va_i, \vy}\,<\, b_i$ for $1\le i\le
  m-1$ (in fact, also for $m+1\le i\le n$). Hence, we can continue
  with
  \begin{align*}
    \sum\eta_ib_i\ =\ b_m\ =\ \sprod{\va_m, \vy}\ =\
    \sum\eta_i\sprod{\va_i, \vy}\ <\ \sum\eta_ib_i\,.
  \end{align*}
  This is clearly a contradiction, so $\va_m$ is not a convex
  combination of the other $\va_j$.
\end{proof}
\begin{figure}[t]
  \centering
  \subfigure[The polytope $\Acore$ in $(\R^2)^*$\label{fig:acore}]{%
    \begin{tikzpicture}[scale=.65]
      \tikzstyle{dot} = [draw,line width=1pt,black,fill=col2!60] 
      \foreach \x in {-1,...,3} { 
        \foreach \y in {-1,...,3} { 
          \coordinate (a\x\y) at (\x,\y); 
          \draw[dot] (a\x\y) circle (4pt);
        } 
      } 
    \tikzstyle{dot} = [draw,line width=1pt,black,fill=col2!60] 
    \tikzstyle{vert} = [draw,line width=1pt,black,fill=col3!60] 
    \tikzstyle{bdedge} = [draw,line width=3pt,black,join=round] 
    \draw[bdedge] (a10) -- (a12);
    \draw[dot] (a11) circle (4pt);
    \foreach \x/\y in {1/0,1/2} {
      \draw[vert] (a\x\y) circle (5pt);
    }
      
    \draw (-3,0) circle (0pt);
    \draw (4,0) circle (0pt);
    \end{tikzpicture}\hspace*{.5cm}
  } \subfigure[The linear space $L_\setA$ and some translates in
  direction $\1$ containing integral points\label{fig:linspan}]{%
    \begin{tikzpicture}[scale=.5]
      \tikzstyle{dot} = [draw,line width=1pt,black,fill=col2!60] 
      \foreach \x in {-3,...,5} { 
        \foreach \y in {-3,...,4} { 
          \coordinate (a\x\y) at (\x,\y); 
        } 
      } 
    \tikzstyle{dot} = [draw,line width=1pt,black,fill=col2!60] 
    \tikzstyle{vert} = [draw,line width=1pt,black,fill=col3!60] 
    \tikzstyle{bdedge} = [draw,line width=2pt,black,join=round] 
    \tikzstyle{bdedgeg} = [draw,line width=1pt,gray,join=round] 
    \tikzstyle{axis} = [draw,line width=1pt,gray,->,join=round] 
    \tikzstyle{vec} = [draw,line width=2pt,col1,->] 

    \draw[axis] (a0-3) -- (a04);
    \draw[axis] (a-30) -- (a40);
    \draw[bdedge] (a3-3) -- (a-33);
    \draw[bdedgeg] (a3-2) -- (a-23);
    \draw[bdedgeg] (a3-2) -- (a-23);
    \draw[bdedgeg] (a4-2) -- (a-13);
    \draw[bdedgeg] (a4-1) -- (a03);
    \draw[vec] (a00) -- (a-11);

      \foreach \x in {-2,...,4} { 
        \foreach \y in {-2,...,3} { 
          \draw[dot] (a\x\y) circle (4pt);
        } 
      } 
    \draw (-6,0) circle (0pt);
    \draw  (8,0) circle (0pt);
    \draw[vert] (a00) circle (4pt);
    \end{tikzpicture}
  }
  \caption{The polytope $\Acore$ and the linear space $L_\setA$ for
    the example in Figure~\ref{fig:core}.}
  \label{fig:acoreandlinspan}
\end{figure}
From now on we fix some parameter $\alpha>0$ and restrict to lattice
polytopes in $\pccan(d)$.  We subdivide this set according to the
configuration of primitive normal vectors spanning the core face of a
polytope in $\pccan(d)$, \textit{i.e.}\ according to the polytope
$\Acore$ it generates. To this end let $\setA:=\{\,\va_1, \ldots,
\va_m\,\}\subset (\Z^d)^*$ be a set of primitive vectors that
positively span the linear space they generate (\textit{i.e.}\ there
are $\lambda_1, \ldots, \lambda_m\ge 0$ and not all $0$ such that
$\sum\lambda_i\va_i=\0$).  For such a set $\setA$ and parameters
$d\in\N$ and $\alpha>0$ we consider the set
\begin{align*}
  \nclass(d, \alpha, \setA)\ :=\ \left\{\,P\, \mid
    P\,\in\,\pccan(d)\text{ and } \cf(P)\;=\; \setA\right\}\,.
\end{align*}
Note that we do not require that the vectors in $\setA$ are the
generators of rays of a subfan of $\Sigma_P$. We just assume that they
generate rays in $\Sigma_P(1)$ and are normal to $\Caff(P)$. We will
show that, up to lattice equivalence and for fixed $d\in\Z_{>0}$ and
$\alpha>0$, only finitely many of the sets $\nclass(d,\alpha,\setA)$
are non-empty. This is obtained by reducing the problem to the
following result of Lagarias and Ziegler.
\begin{theorem}[{Lagarias \&
    Ziegler~\cite{MR1138580}}]\label{thm:lagariasZiegler} Let integers
    $d, k, r\ge 1$ be given. There are, up to lattice equivalence,
    only finitely many different lattice polytopes of dimension $d$
    with exactly $k$ interior points in the lattice $r\Z^d$.\qed
\end{theorem}
This theorem is based on work of Hensley~\cite{Hensley1983}, who
proved this for $r=1$. The bound has later been improved by
Pikhurko~\cite{1045.52009}.  The proof of Lagarias and Ziegler has two
steps. They first show that the volume of a $d$-dimensional lattice
polytope with exactly $k$ interior points in $r\Z^d$ is bounded. In a
second step, they show that any polytope with volume $V$ can be
transformed into a polytope inside a cube with volume $d!V$.

The next lemma gives the reduction to this theorem by showing that for
$P\in \pccan(d)$ the origin is the only relatively interior point of
$\Acore$ in the lattice $\lceil\frac 1\alpha\rceil\Z^d$. In other
words, this implies that $\Acore$ is $\Q$-Fano if it corresponds to a
lattice polytope in $\pccan(d)$.
\begin{lemma}\label{lem:spectrumcanonical}
  For $\Acore$ and $\alpha$ as above we have $\relint\,
  \alpha\Acore\;\cap\; (\Z^d)^*\; =\; \{\0\}$.
\end{lemma}
\begin{proof}
  We prove this by contradiction. So assume that there is some vector
  $\va\in (\Z^d)^*\setminus\{\0\}$ contained in the relative interior
  of $\alpha\Acore$. As $\0\in \relint\, \alpha\Acore$ the point $\va$
  is contained in the cone spanned by the vertices of some facet $F$
  of $\Acore$. Hence, we can find $\lambda_1, \ldots, \lambda_m\ge 0$
  with $\lambda_i=0$ if $\va_i\not \in F$ and
  \begin{align*}
    \Lambda\ :=\ \sum\lambda_i\,<\, \alpha
  \end{align*}
  such that $\va\, =\, \sum_{i=1}^m\lambda_i\va_i$.  Note that
  $\Lambda$ is strictly less than $\alpha$ as $\va$ is in the relative
  interior of $\alpha\Acore$.  The normal fan $\Sigma_P$ of $P$ is
  complete. Hence, there is some (possibly not unique) maximal cone
  $\sigma\in \Sigma_P(d)$ that contains $\va$. Then
  $\sigma=\cone(\va_i\mid i\in I)$ for some set $I\subseteq[n]$ (where
  $\va_1, \ldots, \va_n$ is the list of \emph{all} normal vectors of
  facets of $P$).  This gives a second representation of $\va$ in the
  form $\va\, =\, \sum_{i\in I}\mu_i\va_i$, where $\mu_i\ge 0$.  By
  assumption, the normal fan $\Sigma_P$ is $\alpha$-canonical,
  \textit{i.e.} the height of $\va$ in $\sigma$ is at least
  $\alpha$. Hence, we can choose the $\mu_i$ in such a way that
  \begin{align*}
    M\ :=\ \sum_{i\in I}\mu_i \ \ge\ \alpha\,.
  \end{align*}
  Note that this implies
  \begin{align}
    \Lambda\ <\ \alpha\ \le \ M\label{eq:LleM}\,.
  \end{align}
  The cone $\sigma\,=\,\cone(\va_i\mid i\in I)$ is a maximal cone in
  $\Sigma_P$. Hence, it is the normal cone of a vertex $\vv$ of
  $P$. By its definition, this implies that $\sprod{\va_i,\vv}=b_i$ if
  $i\in I$ and $\sprod{\va_i,\vv}<b_i$ otherwise.  Thus we compute
  \begin{align}\label{eq:spectrum1}
    \sum_{i=1}^m\lambda_ib_i\ &\ge\ \sum_{i=1}^m\lambda_i\sprod{\va_i,
      \vv}\ =\ \sprod{\va, \vv}\ =\ \sum_{i\in I}\mu_i\sprod{\va_i,
      \vv}\ =\ \sum_{i\in I}\mu_ib_i\,.
  \end{align}

  Now let $\vxcore$ be a relative interior point of $C=\core\,P$. Then
  \begin{alignat*}{4}
    &&\sprod{\va_i,\vxcore}\ &=\ b_i-c &\quad& \text{for}& 1\ \le\ i\ \le\ m\\
    \text{and}&\quad&
    \sprod{\va_i,\vxcore}\ &<\ b_i-c && \text{otherwise}\,.
  \end{alignat*}
  Using this we compute
  \begin{alignat*}{3}
    \sum_{i\in I}\mu_ib_i -\ Mc\ &\ =\ \sum_{i\in I}\mu_i(b_i-c)
    &&\ge\ \sum_{i\in I}\mu_i\sprod{\va_i, \vxcore}\ && =\ \sprod{\va,
      \vxcore}\ =\ \sum_{i=1}^m\lambda_i\sprod{\va_i, \vxcore}\\
    &&&&&=\ \sum_{i=1}^m\lambda_i(b_i-c)\ =\
    \sum_{i=1}^m\lambda_ib_i\ -\ \Lambda c\,.
  \end{alignat*}
  We can shorten this to
  \begin{align*}
    \sum_{i\in I}\mu_ib_i\ &\ge\
    \sum_{i=1}^m\lambda_ib_i\ +\ (M-\Lambda)c\ > \ \sum_{i=1}^m\lambda_ib_i\;,
  \end{align*}
  where the last strict inequality follows from \eqref{eq:LleM}.  This
  contradicts \eqref{eq:spectrum1}.
\end{proof}
We can now combine this result with Theorem~\ref{thm:lagariasZiegler}
to obtain the desired finiteness of configurations with non-empty
$\nclass(d,\alpha,\setA)$.
\begin{corollary}\label{cor:spectrumcanonical}
  Let $d\in\N$ and $\alpha>0$ be given. Then there are, up to
  unimodular transformation, only finitely many sets
  $\setA\subseteq(\Z^d)^*$ such that $\nclass(d,\alpha,\setA)$ is
  non-empty.
\end{corollary}
\begin{proof}
  By Lemma~\ref{lem:spectrumcanonical} $\relint\,\alpha\Acore$ only
  contains the origin. Hence,
  $\relint\,\Acore\cap\left\lceil\frac1\alpha\right\rceil\Z^d$
  contains exactly one point. By Theorem~\ref{thm:lagariasZiegler}
  there are only finitely many possible such configurations, up to
  lattice equivalence.
\end{proof}
In other words, this corollary shows that there are, up to lattice
equivalence, only finitely many configurations that occur as the set
$\cf(P)$ of core normals of a lattice polytope $P\in\pccan(d)$.

Let $P$ be a polytope in $\nclass(d,\alpha,\setA)$ with $\Q$-codegree
$c^{-1}=\qcd(P)$ given by
\begin{align*}
  P\ &:=\ \{\,\vx\,\mid \, \sprod{\va_i, \vx}\, \le \, b_i\quad 1\le
  i\le n\,\}
\end{align*}
for some $b_i$, $n\ge m$, and $\setA=\{\va_1, \ldots, \va_m\}$. Choose
an relative interior point $\vxcore$ in $\core\,P$. Then
$\sprod{\va_i,\vxcore}+c=b_i$ for $1\le i\le m$. The $b_i$ are the
right hand sides of our inequality description of a lattice polytope,
so in particular, $\sprod{\va_i,\vxcore}+c\in \Z$ for $1\le i\le m$.
Let $A$ be the $(m\times d)$-matrix whose rows are the $\va_j$, $1\le
j\le m$. Let $\1\in \R^m$ be the all-ones vector. Then, in matrix
form, this reads
\begin{align*}
  A\vxcore\,+\,c\1\in \Z^m\,.
\end{align*}
In other words, whenever $\setA$ is the set of core normals of a
lattice polytope with $\Q$-codegree $q$, then there is a rational
point $\vy$ such that
\begin{align}\label{eq:spectrumneccond}
  A\vy\,+\,c\1\in \Z^m\,.
\end{align}
Hence, we have found a necessary condition that any pair $\setA$ and
$c$ must satisfy if they come from a lattice polytope $P\in \nclass(d,
\alpha, \setA)$. We use this to study possible values for $c$.
\begin{lemma}\label{prop:spectrumcanonical}
  Let $\alpha,\eps>0$, $d\in \N$ and $\setA$ as above.  Then
  \begin{align*}
    \left\{\,\qcd(P)\,\mid \, P\in \nclass(d, \alpha, \setA)\text{ and
      } \qcd(P)\ge \eps\,\right\}
  \end{align*}
is finite.
\end{lemma}
\begin{proof}
  Let $A$ be the matrix with rows given by the vectors in $\setA$ as
  above, and let $\ov\va_k$, $1\le k\le d$ be the columns of $A$. Let
  $L_\setA\subseteq\R^m$ be the linear span of $\ov\va_1, \ldots,
  \ov\va_d$. See Figure~\ref{fig:acoreandlinspan}.\subref{fig:linspan}.

  We will now show that \eqref{eq:spectrumneccond} has only finitely
  many possible solutions for $c$ in the range $0<c\le \eps^{-1}$
  (but, of course, for any given $c$ such
  that~\eqref{eq:spectrumneccond} has at least one solution there are
  infinitely many solutions $\vy$ for this fixed $c$).  By assumption,
  the vectors $\va_1, \ldots, \va_m\in \setA$ positively span their
  linear span. Hence, $\rank(\va_1, \ldots, \va_m)\le m-1$ and
  $L_\setA$ is a proper subspace of $\R^d$. Further, as there are
  $\lambda_1, \ldots, \lambda_m\ge 0$ and not all zero such that
  $\0=\sum\lambda_i\va_i$, there is no $\vw\in \R^d$ such that
  $\sprod{\va_j,\vw}>0$ for all $1\le j\le m$.  In other words,
  $L_\setA$ meets the interior of the positive orthant of $\R^m$ only
  in the origin. In particular, $\1\not\in L_\setA$.

  Thus, there are only finitely many translates of the form
  $L_\setA+c\1$ for $0< c\le \eps^{-1}$ that contain a lattice point
  (namely those that are at distance $\frac p{\det \ov M}$ in
  direction $\1$, where $M$ is the row vector matrix of any lattice
  basis of $L_\setA$ and $\ov M=(M,\1)$ the matrix $M$ with an
  additional column of $\1$). Put differently, the set
  \begin{align*}
    \left\{c\mid A\vx\;+\; c\1\in \Z^m\text{ for some
        $\vx\in\R^d$ and $0<c\le \eps^{-1}$}\;\right\}
  \end{align*}
  is finite.  This proves the claim.
\end{proof}
Combining this with Corollary~\ref{cor:spectrumcanonical} finally
proves Theorem~\ref{thm:spectrumcanonical}.
\begin{remark}
  The proof of Theorem~\ref{thm:spectrumcanonical} also gives a way to
  explicitly compute a (finite superset of) the possible values of
  $\qcd(P)$. Indeed, all possible values are contained in the set
  $\left\{\,p(\det \ov M)^{-1}\;\,\middle|\,\; p\in
    \Z_{>0}\,\right\}$, where $M$ is any matrix whose rows are a
  lattice basis of the lattice in the linear subspace $L_\setA$ of the
  previous lemma and $\ov M$ this matrix with an additional
  column of ones. 
\end{remark}

\end{document}